\documentclass[12pt, reqno]{amsart}

\usepackage{amsmath, amsthm, amscd, amsfonts, amssymb, graphicx, color}
\usepackage[bookmarksnumbered, colorlinks, plainpages]{hyperref}

 \makeatletter \oddsidemargin.9375in
\evensidemargin \oddsidemargin \marginparwidth 2in \makeatother

\newtheorem{theorem}{Theorem}[section]
\newtheorem{definition}[theorem]{Definition}
\newtheorem{cor}[theorem]{Corollary}
\newtheorem{lem}[theorem]{Lemma}
\newtheorem{prop}[theorem]{Proposition}

\newtheorem{remark}[theorem]{Remark}
\numberwithin{equation}{section}

\newcommand{\ch}{{\rm ch}}

\begin{document}
%
\title[Multiplicative spectral functions on some Banach function algebras]{Multiplicative spectral functions on some Banach function algebras}

\author{Nahid Bayati and Fereshteh Sady}

\address{Department of Pure Mathematics, Faculty of  Mathematical Sciences, Tarbiat Modares University, Tehran, 14115-134, Iran}

\vspace*{.25cm}

\email{nahid.bayati@modares.ac.ir, \, sady@modares.ac.ir}


\maketitle


%
\begin{abstract}
In this paper, we study multiplicative functions $\varphi  \colon    A \to \Bbb C$ on a natural Banach function algebra $A$ on a compact Hausdorff space $X$,  such that  $\varphi(f)\in \sigma(f)$ for all $f\in A$. It is shown that for certain natural Banach function algebras $A$, either $\ker(\varphi)$ is a maximal ideal of $A$ or $1\in {\rm span}({\rm ker}(\varphi))$ (that is $1=f_1+f_2+\cdots f_n$ for some $f_1,..., f_n \in {\rm ker}(\varphi)$).   Then we investigate for the linearity of $\varphi$ in either of cases that $\varphi$ is continuous or $1\notin {\rm span}({\rm ker}(\varphi)$. We show that, for some natural Banach function algebras $A$, in either of these cases,   there exists a point $x_0\in X$ such that  $\varphi(f)=f(x_0)$ for some  family of functions $f\in A$ (including those functions $f\in A$ that $\overline{f}\in A$). In particular, such a multiplicative spectral function on some Banach algebras including $C(X)$, Lipschitz algebras, Banach algebras of absolutely continuous functions on $[0,1]$ and $C^1([0,1])$ is linear and hence it is a character.   
\end{abstract}

\maketitle
\section{Introduction}
For a unital complex Banach algebra $A$, let $G(A)$ be the group of all invertible elements of $A$,  $G_1(A)$ be the principle component of $G(A)$ and ${\rm exp}(A)$ be the subset  $\{e^a: a\in A\}$ of $G_1(A)$.  Given  $a\in A$,  let $\sigma(a)$ denote its  spectrum.   We say that a function (not necessarily linear) $\varphi \colon  A \to \Bbb C$ on a Banach algebra $A$ is spectral, if $\varphi(a) \in \sigma(a)$ for all $a\in A$. By the classical Gleason-Kahane-\.Zelazko theorem, any complex-linear spectral function $\varphi  \colon    A \to \Bbb C$ on a unital Banach algebra $A$ is a character on $A$.  The Kowalski-S\l odowski theorem \cite{Kow} provides an additive version of the Gleason-Kahane-\.Zelazko theorem by showing that any complex function $\varphi  \colon    A \to \Bbb C$ that satisfies $\varphi(a)+\varphi(b) \in \sigma(a+b)$ for all $a,b \in A$, is a character on $A$. Without the linearity assumption, the following result has been proven by A. Maouche for multiplicative spectral functions on a unital Banach algebra.  

 \begin{theorem} \cite{Ma} \label{Ma-th}
 Let $A$ be a (complex) unital Banach algebra and $\varphi  \colon    A \to \Bbb C$ be a multiplicative function such that $\varphi(a)\in \sigma(a)$ for all $a\in A$. Then there exists a unique character $\psi$ on $A$ such that $\psi(a)=\varphi(a)$ for all  $a\in G_1(A)$.  
 \end{theorem}
 The example given in \cite{Ma} shows that the above multiplicative spectral function $\varphi$ is not necessarily linear. However, under the continuity assumption,  there are some recent results concerning the linearity of such multiplicative functions. For a compact Hausdorff space $X$, it was shown in  \cite{Brits-CX} any continuous multiplicative spectral function on the Banach algebra $C(X)$ of all continuous complex-valued  functions on $X$, is a character and, in particular,  it is linear.   More generally, in \cite{Brits-Cstar} the same result has been obtained for continuous multiplicative spectral functions on  arbitrary $C^*$-algebras.  Finally, it was shown in \cite{Brits-Her}, that if $A$ is a Hermitian Banach algebra, then any continuous multiplicative spectral function $\varphi   \colon    A \to \Bbb C$ is a character, as well.   We note that such results  can be considered as  multiplicative versions of the Gleason-Kahane-\.Zelazko  theorem. We also refer to \cite{Brits-pre}, in which a  multiplicative Kowalski-S\l odkowski theorem  has been proven for a Hermitian   algebra $A$, i.e. it was shown  that  if $\varphi: A \to \Bbb C$ is a continuous function satisfying $\varphi(x) \varphi(y) \in \sigma(xy)$ for all $x,y\in A$, then either $\varphi$ or $-\varphi$ is a character on $A$.

We should note that for a compact Hausdorff space $X$, it was shown in \cite{Bayati-Sady} that for a  multiplicative spectral function $\varphi: C(X) \to \Bbb C$  (not assumed to be continuous), either ${\rm ker}(\varphi)$ is a maximal ideal or $1\in {\rm span}({\rm ker}(\varphi)) $  and  $\varphi$ is linear in either of cases that $\varphi$ is continuous or $1\notin {\rm span}({\rm ker}(\varphi)) $. 

In this paper,  we improve the results of \cite{Bayati-Sady}. We consider the case that $A$ is a natural Banach function algebra on a compact Hausdorff space $X$ and study multiplicative spectral functions $\varphi  \colon    A \to \Bbb C$.  In the lack of continuity assumption, we first show that for certain natural Banach function algebras $A$, either  $\ker(\varphi)$ is a maximal ideal of $A$ or it spans $A$, i.e.  $1=f_1+\cdots f_n$ for some $f_1,..., f_n\in \ker(\varphi)$. 
 Moreover, in the uniform algebra case, there exists a point $x_0\in X$ such that $\varphi(f)=f(x_0)$ for all $f \in G(A)$ (Proposition \ref{Main-P}).  Then we investigate for the linearity of $\varphi$ whenever either $1\notin \ker(\varphi)$ or $\varphi$ is continuous. Indeed, we prove two theorems 
 (Theorems \ref{Main1} and \ref{Main2})   providing a point $x_0\in X$ such that $\varphi(f)=f(x_0) $ for some family of functions $f\in A$, including those functions $f\in A$ with  $\overline{f}\in A$.  The first theorem concludes  a result concerning the linearity of multiplicative spectral functions on the Banach algebras such as $C(X)$, Lipschitz algebras and also the Banach algebra of all absolutely continuous functions on $[0,1]$. The second theorem provides another  linearity result for some Banach function algebras such as  $C^1([0,1])$.    Our approach is based on using peaking functions, strong boundary points  and their properties in uniform algebras and Banach function algebras.

\section{Preliminaries}
For a compact Hausdorff space $X$, we denote the Banach algebra of all continuous  complex-valued functions on $X$ by $C(X)$. The supremum norm of a function $f\in C(X)$ will be denoted by $\|f\|_X$. A  point separating subalgebra  $A$ of $C(X)$  which contains the constant functions is called a Banach function algebra on $X$ if it is a Banach algebra under some norm $\|\cdot\|$.   In the case that the norm of a Banach function algebra $A$ on $X$ is the same supremum norm, we call it a uniform algebra on $X$. For a Banach function algebra $A$ on $X$, $\overline{A}$ is the uniform closure of $A$.  
It is easy to see that the norm $\|\cdot \|$ of a Banach function algebra on $X$ satisfies the inequality $\|\cdot \|\ge \|\cdot \|_X$.   A Banach function algebra $A$ on $X$ is called natural if  each character on $A$ is an evaluation homomorphism at some point of $X$. 

For a subspace $A$ of $C(X)$, let $A^*$ denote its dual space (with respect to the supremum norm).  For each $x\in X$, $\varphi_x$ is the evaluation functional on $A$ at $x$ and the Choquet boundary $\ch(A)$ of $A$ consists of all points $x\in X$ such that $\varphi_x$ is an extreme point of the closed unit ball of $A^*$.  It is well-known that $\ch(A)$ is a boundary for $A$, that is for each $f\in A$ there exists a point $x_0\in \ch(A)$ such that $|f(x_0)|=\|f\|_X$. A point $x_0\in X$ is called a strong boundary point of $A$ if for each neighborhood $U$ of $x_0$ and $\varepsilon>0$  there exists a function $f\in A$ such that $f(x_0)=1=\|f\|_X$ and $|f(x)|\le \varepsilon$ for all $x\in X\setminus U$.   
   We denote the set of all strong boundary points of $A$ by $\Theta(A)$. 
It should be noted  that $\Theta(A)\subseteq \ch(A)$   and in the  uniform algebra case  we have $\Theta(A)=\ch(A)$ \cite[Theorems 2.2.1, 2.3.4]{Brow}. 

Clearly for a compact Hausdorff space $X$, the uniform algebra $C(X)$ is natural and its Choquet boundary is the same $X$.  For an example of a nontrivial natural uniform algebra $A$ on a compact (metric) space $X$ with $\ch(A)=X$ we can refer to the Cole example (see \cite[Page 255]{Brow}). Here are some examples of natural Banach function algebras $A$ with $\Theta(A)=\ch(A)=X$: 

(i) The Banach algebra ${\rm Lip}(X)$ consisting of all  complex Lipschitz functions on a compact metric space $X$ endowed with the norm $\|f\|=\|f\|_X+p(f)$ where $p(f)$ is the Lipschitz constant of $f\in {\rm Lip}(X)$. 

(ii) The Banach algebra $C^1([0,1])$ of all continuously differentiable functions on the unit interval $[0,1]$ endowed with the norm $\|f\|=\|f\|_{[0,1]}+\|f'\|_{[0,1]}$. 

(iii) The Banach algebra ${\rm AC}([0,1])$ of all absolutely continuous functions on $[0,1]$ under the norm 
$\|f\|=\|f\|_{[0,1]}+{\rm var}(f)$ where ${\rm var}(f)$ is the total variation of $f\in {\rm AC}([0,1])$. 

Let $A$ be a subspace of $C(X)$. A function $f\in A$  is called a peaking function of $A$, if $1\in f(X)\subseteq  {\rm int}(\Bbb D) \cup \{1\}$ where $\Bbb D$ is the closed unit disk in the complex plane. 
For $x_0\in A$, we use the notation $P_{x_0}(A)$ for  the set of all peaking functions $f \in A$ with $f(x_0)=1$. Such a function $f$ is called a peaking function at $x_0$. Hence for a peaking function $f\in A$ at $x_0$, its  maximum set $M(f)=\{x\in X:  |f(x)|=\|f\|_X\}$ is a closed subset of $X$  containing $x_0$ such that  $|f(y)|< 1$ for all $y\in X \setminus M(f)$.   A  point $x_0\in X$ is called  a peak point of $A$ if there exists a peaking function $f\in P_{x_0}(A)$ with $M(f)=\{x_0\}$. 

We should note that if the subspace $A$ of $C(X)$ contains the constant function 1, then for each $\varepsilon>0$,  $x_0\in \Theta(A)$ and open neighborhood $U$ of $x_0$,  replacing $f$ by $\frac{1+f}{2}$, we can find a function $f\in P_{x_0}(A)$ such that $|f| < \varepsilon$ on $X \setminus U$.


The following lemma is easily verified. For the sake of completeness we state its proof. 
\begin{lem}\label{eq-inf}
Let $A$ be a Banach function algebra on $X$. Then  for each $x_0\in \ch(A)$ and $f\in A$ we have 
\begin{equation}\label{st-inf} 
 |f(x_0)|=\inf\{\|fh\|_{X}: h\in P_{x_0}(\overline{A})\cap {\rm exp}(\overline{A})\}
\end{equation}
\end{lem}
  \begin{proof}
  The equality is obvious if $f=0$. So we assume that $f$ is nonzero. 
Clearly for each $h\in P_{x_0}(\overline{A})\cap {\rm exp}(\overline{A})$ we have $|f(x_0)|\le \|fh\|_X$. Now assume  that  $x_0\in \ch(A)$. Given $\epsilon\in (0,1)$, we set $U=\{x\in X: |f(x)|< |f(x_0)|+ \epsilon\}$. Clearly $U$  is a neighborhood of $x_0$.  We note that $ \ch(A)=\ch(\overline{A})$ and since $\overline{A}$ is a uniform algebra on $X$ we have $\ch(\overline{A})=\Theta(\overline{A})$.  Hence  there exists a function $u\in P_{x_0}(\overline{A})$ such that  $|u|<\epsilon$ on $X\setminus U$. Then for  any $n\in \Bbb N$ with $e^{-n(1-\varepsilon)}< \frac{\epsilon}{\|f\|_X}$,  the function $h=e^{-n (1-u)}$  is an element of $ P_{x_0}(\overline{A})\cap {\rm exp}(\overline{A})$ such that   $|h|< \frac{\epsilon}{\|f\|_X}$ on $X\setminus U$. An easy verification shows that 
$\|fh\|_X \le |f(x_0)|+ \epsilon$, which proves the desired equality.  
   \end{proof}
The following lemma, called multiplicative Bishop's lemma, has been proven in \cite{Hat}. 
\begin{lem}\label{Mult-Unif}
Let $A$ be a uniform algebra on a compact Hausdorff space $X$ and $x_0 \in \ch(A)$. Then for each $f\in A$ with $f(x_0)\neq 0$ there exists $g\in P_{x_0}(A)$ such that $\frac{1}{f(x_0)}fg \in P_{x_0}(A)$. 
\end{lem} 

\begin{remark}\label{Rem-1} {\em 
i) By \cite[Page 283]{Hat}, we may choose the desired function $g$ in the above lemma in such a way that $g$ is  an element of ${\rm exp}(A)$. 

ii) We should note that the closedness of $A$ in $C(X)$ in the above  lemma has an important role in its  proof. However, as the next lemma shows,  in some Banach function algebras like the Banach algebra of Lipschitz functions on a compact metric space or absolutely continuous functions on the unit interval,  we may obtain directly the function $g$ with the desired properties.   
}
\end{remark}
\begin{lem} \label{Mult-Lip}
Let $A$ be the Banach function algebra ${\rm Lip}(X)$ for some compact metric space $(X,d)$.   Then for each point  $x_0\in X$  and $f\in A$ with $f(x_0)\neq 0$ there exists a nonnegative peaking function $g\in P_{x_0}(A)$ such that $\frac{1}{f(x_0)}fg \in P_{x_0}(A)$. The same conclusion holds for $A={\rm AC}(X)$ where $X=[0,1]$. 
\end{lem} 
\begin{proof}
First consider the case that $A={\rm Lip}(X)$. We may assume that $f(x_0)=1$.  Given $x_0\in X$, the function $h_{x_0}  \colon    X \to  [0,1]$ defined by $h_{x_0}(x)= 1- \frac {d(x, x_0)}{{\rm diam}(X)}$ is a peaking function of $A$ with $M(h_{x_0})=\{x_0\}$, that is $x_0$ is a peak point of $A$. We can choose easily  a function $g_0\in A$  with values in $[0,1]$ such that $g_0=1$ on the set $\{x\in X: |f(x)|\le 1\}$ and $g_0=0$ on $\{x:  |f(x)| \ge 2\}$. Now let  $g:X \to [0,1]$ be defined by $g={\rm min } (2-|f(x)|,1) h_{x_0} g_0$. Then an easy verification shows that $g$ is a peaking function in $A$ satisfying the desired properties.   

The same proof works for  $A={\rm AC}([0,1])$, since any Lipschitz function is absolutely continuous. 
\end{proof}

\section{Main Result}
 Throughout this section, unless otherwise is stated, we assume that $A$ is a natural Banach function algebra on a compact Hausdorff space $X$.  We put  $A_+=\{f\in A: f\ge 0\}$. As we noted before, by a spectral function on $A$ we mean a (not necessarily continuous)  function  $\varphi \colon A \to \Bbb C$ satisfying   $\varphi(f) \in \sigma(f)$ for all $f \in A$.  Clearly in this case we have $\sigma(f)=f(X)$ for all $f\in A$. By   Theorem \ref{Ma-th}  for each multiplicative spectral function $\varphi \colon  A\to \Bbb C$ there exists a unique point $x_0\in X$ such that $\varphi(e^f)=e^{f(x_0)}$ for all $f \in A$. We call $x_0$ the associated point of $\varphi$. 
 
 We begin with the following proposition concerning the kernel of a  multiplicative spectral function on $A$.

\begin{prop}\label{Main-P}
Let   $\ch(A)=X$  and    $\varphi:A \to \Bbb C$ be a (not necessarily continuous) multiplicative spectral function.  Then the following statements hold. 

(i) There exists a unique  point $x_0\in X$ such that 
$|\varphi(f)|\le |f(x_0)|$ for all $f\in A$. In particular,  $|\varphi(f)|=|f(x_0)|$ for all $f\in G(A)$. 

(ii) Either $\ker(\varphi)$ is a maximal ideal of $A$ or  $1=f_1+ \cdots +f_n$ for some $f_1, ..., f_n \in \ker(\varphi)$.  

(iii) In the case that  $A$  is a uniform algebra, we have  $\varphi(f)=f(x_0)$ for all $f \in G(A)$. 
\end{prop}
\begin{proof}
i) Let $x_0\in X$ be the associated point of $\varphi$, i.e.   $\varphi=\varphi_{x_0}$ on ${\rm exp}(A)$. We show that $|\varphi(f)|\le |f(x_0)|$ for all $f\in A$. Given $f\in A$ and $\epsilon>0$, by Lemma \ref{Mult-Unif},  there exists a function $h\in P_{x_0}(\overline{A})\cap {\rm exp}(\overline{A})$ such that $\|f h\|_X < |f(x_0)|+\epsilon$. Choose a sequence $\{h_n\}$ in ${\rm exp}(A)$ which uniformly converges to $h$. Then we have  $\|fh_n\|_X<|f(x_0)|+\epsilon$ for sufficiently large $n\in \Bbb N$.    
 Clearly  we have 
$\varphi(h_n)=h_n(x_0) \to 1$. Hence, using the spectral condition once again, we have 
\[|\varphi(f) h_n(x_0)|=|\varphi(fh_n)|\le \|fh_n\|_X< |f(x_0)|+\epsilon,\]
 for sufficiently large $n\in \Bbb N$. Tending $n \to \infty$, since $\epsilon>0$ is arbitrary  we get $|\varphi(f)|\le |f(x_0)|$, as desired. 

 Now assume that $f\in G(A)$. Without loss of generality we assume that $f(x_0)=1$. Since 
$|\varphi(f)|\le |f(x_0)|=1$, $|\varphi(f^{-1})|\le |f^{-1}(x_0)|=1$ and $\varphi(ff^{-1})=1$ we get 
$|\varphi(f)|=1=|f(x_0)|$.

The uniqueness property is immediate from the equality $|\varphi(e^f)|=e^{{\rm Re}(f(x_0))}$ for all $f\in A$, since $A$ separates the points of $X$. 

We note that, by  the above argument, $x_0$ is indeed, the same associated point to $\varphi$, that is $\varphi(e^f)=e^{f(x_0)}$ for all $f\in A$. 

ii) Clearly,  by part (i), we have   $\varphi(f)=f(x_0)$ for all $f\in A$ with $f(x_0)=0$, that is ${\rm ker}(\varphi_{x_0})\subseteq \ker(\varphi)$. Being $\varphi$ multiplicative, it follows that ${\rm span}(\ker(\varphi))$ is an ideal of $A$. Hence we have either $1\in {\rm span}(\ker(\varphi))$ or   ${\rm span}(\ker(\varphi)) =  \ker(\varphi_{x_0})$. In the latter case, we have clearly  $\ker(\varphi)=\ker(\varphi_{x_0})$, i.e. $\ker(\varphi)$ is a maximal ideal of $A$.

iii) Let $A$ be a uniform algebra and $f\in G(A)$. We again assume that $f(x_0)=1$.   Then by Lemma \ref{Mult-Unif} and Remark \ref{Rem-1},  there exists a function $h\in P_{x_0}(A)\cap {\rm exp}(A)$ such that  $fh\in P_{x_0}(A)$. Clearly $\varphi(h)=h(x_0)=1$.  Being $fh$ invertible in $A$, we have $|\varphi(fh)|=|fh(x_0)|=1$ and since $\varphi(fh) \in fh(X) \subseteq \Bbb {\rm int}(\Bbb D) \cup \{1\}$  we get   $\varphi(fh)=1$.  Thus $\varphi(f)=\varphi(fh) \varphi(h^{-1})=1=f(x_0)$. 
\end{proof}
Next corollary gives a generalization of part (iii) in the above proposition. 
 \begin{cor}\label{cor-zero}
   Under the assumptions of Proposition \ref{Main-P}(iii), we have $\varphi(f)\in\{0,f(x_0)\}$ for each $f\in A$ such that $0$ is not a limit point of $f(X)$. 
     \end{cor}
    \begin{proof}
    We first show  that $\varphi(h)\in \{0,h(x_0)\}$ if $h\in P_{x_0}(A)$ and  $0$ is not a limit point of $h(X)$. To do this, we note that for each $n\in \Bbb N$ we have $\varphi(e^{n(1-h)})=1$ and consequently, by hypotheses, we have 
    $$\varphi(h)=\varphi\big(h\, e^{n(1-h)}\big)=h(x_n) e^{n(1-h(x_n))} \qquad (n\in \Bbb N),$$ 
    for some sequence $\{x_n\}$ in $X$. 
 Passing through a subsequence, we assume that $h(x_n) \to c$ for some $c\in h(X) \subseteq {\rm int}(\Bbb D) \cup \{1\}$. If $c=0$, then since $0$ is not a limit point of $h(X)$, we have  $h(x_n)=0$ for all sufficiently large $n\in \Bbb N$. In particular, $\varphi(h)=0$ in this case. So assume that $c\neq 0$. Clearly, in this case the sequence $\{ e^{n(1-h(x_n))}\}$  is convergent.  
 If $c\neq 1$ (equivalently ${\rm Re}(c)\neq 1$), then there exists $r>0$ such that $1-{\rm Re}(h(x_n))>r$ and $|h(x_n)|>r$ for all sufficiently large $n\in \Bbb N$. Thus  
 $$|\varphi(h)| = |h(x_n)| e^{n\big(1-{\rm Re}(h(x_n))\big)}\ge r e^{nr}$$
  for all sufficiently large $n\in \Bbb N$, which is impossible. Hence $c=1$ and consequently 
  $|\varphi(h)|=  \lim_{n \to \infty} e^{n \big(1-{\rm Re}(h(x_n))\big)}\ge 1$. Therefore $|\varphi(h)|=1$ which concludes that $\varphi(h)=1$, since $h$ is a peaking function. 
  
  Now assume that $f\in A$  such that $0$ is not a limit point of $f(X)$. Since $|\varphi(f)|\le |f(x_0)|$, by Proposition \ref{Main-P}, we have $\varphi(f)=0$ if $f(x_0)=0$. So we may assume that $f(x_0)=1$. Then, by Lemma \ref{Mult-Unif}, there exists a function $k\in P_{x_0}(A)\cap {\rm exp}(A)$ such that $fk\in P_{x_0}(A)$. Obviously we have $\varphi(f)=\varphi (fk)$. Set $h=fk$. We claim that $0$ is not a limit point of $h(X)$,  as well. Indeed, assume that  there exists a sequence $\{h(x_n)\}$ of distinct points in $h(X)$ converging to $0$. Passing through a subsequence we may assume that $\{f(x_n)\} $ converges to a point $c\in f(X)$. If $c\neq 0$, then    $k(x_n)\to 0$, that is $0\in k(X)$ which is impossible, since $k\in {\rm exp}(A)$. Hence $f(x_n)\to 0$ and since $0$ is not a limit point of $f(X)$ we get $f(x_n)=0$ for sufficiently large $n\in \Bbb N$. Thus $h(x_n)=0$ for sufficiently large $n\in \Bbb N$. This shows that $0$ is not a limit point of $h(X)$. Therefore, by the above argument, we have $\varphi(f)=\varphi(fk)\in \{0, fk(x_0)\}=\{0, 1\}$, as desired. 
 \end{proof}

\begin{lem}\label{lem-real}
Let $\varphi \colon  A\to \Bbb C$ be a multiplicative spectral function on the natural Banach function algebra $A$ and $x_0\in X$ be its associated point. Then for each real-valued $f\in A$  we have $\varphi(f)\in \{0,f(x_0)\}$. 
\end{lem}
\begin{proof}
Similar to  \cite{Brits-2} we can show that  for each real-valued $f\in A$ with $\varphi(f)\neq 0$ we have $\varphi(1+if)=1+i \varphi(f)$ and 
$\varphi(e^f)=e^{\varphi(f)}$. To do this, first we can choose by hypotheses  $\alpha, \beta \in f(X)$ such that 
\[\varphi(f)(1+i\alpha)= \varphi(f) \varphi(1+if)= \varphi(f (1+if))= \beta(1+i\beta).\]
Hence  $\varphi(f)=\beta$ and $\varphi(f) \alpha=\beta^2$, and consequently  $\varphi(f)=\alpha$. Thus $\varphi(1+if)=(1+i \varphi(f))$. On the other hand, we can also choose a scalar $\gamma\in f(X)$ such that 
\[ e^\gamma (1+i\gamma)= \varphi(e^f( 1+if))= \varphi(e^f) (1+i \varphi(f))\]
which concludes that $\varphi(e^f)=e^\gamma$ and $\varphi(f)=\gamma$, that is $\varphi(e^f)=e^{\varphi(f)}$. 
 Since $\varphi(e^f)=e^{f(x_0)}$ and $f$ is real-valued we have $\varphi(f)=f(x_0)$. 
\end{proof}
\begin{remark}\label{rem-2}{\rm  If $A$ is a natural uniform algebra on $X$, then we can use  Lemma \ref{Mult-Unif} to see  that a  multiplicative spectral function $\varphi \colon  A \to \Bbb C$ is linear (equivalently $\varphi=\varphi_{x_0}$ on  $A$, where $x_0\in X$ is the associated point of $\varphi$) if and only if $|\varphi(h)|=1$  for all $h\in P_{x_0}(A)$. The authors do not aware if (under the continuity assumption on  $\varphi$) this equality holds for all peaking functions or not. However,  for some  family of  peaking functions $h\in P_{x_0}(A)$, we may have  $\varphi(h)\in \{0,1\}$  (see  Corollary \ref{cor-zero} and Lemma \ref{lem-real}).  For another  example,   it is easy to see that if $h\in P_{x_0}(A)$ such that $1$ is not a limit point of $h(X)$, then $\varphi(h)\in \{0, 1\}$. 
Fo this, let  $n\in \Bbb N$. Then since  $1+\frac{n-1}{n}h\in {\rm exp}(A)$ we can  choose a point $x_n\in X$ such that 
  \[ \varphi(h) \Big( 1+ \frac{n-1}{n} \Big)^n=\varphi\Big (h \big(1+\frac{n-1}{n}h\big)^n\Big)=h(x_n) \Big(1+\frac{n-1}{n} h(x_n)\Big)^n.\]
  Hence  for each $n\in \Bbb N$ we have 
  \[\varphi(h)=h(x_n) \Big(\frac{n+(n-1)h(x_n)}{2n-1}\Big)^n.\]
   Passing through a subsequence, there exists a point $c\in h(X)$ such that $h(x_n)\to c$. If $c=1$, then since $1$ is not a limit point of $h(X)$ we get $h(x_n)=1$ for all sufficiently large $n\in \Bbb N$. In particular $\varphi(h)=1$. If $c\neq 1$, then there exists $r\in (0,1)$ such that $|h(x_n)|\le r$ for sufficiently large $n\in \Bbb N$. 
  Choosing $s\in (\frac{r+1}{2}, 1)$ it is easy to see that $\frac{n+(n-1)r}{2n-1}<s$ whenever $n\in \Bbb N$ is sufficiently large.   
  This implies that $   \big(\frac{n+(n-1)h(x_n)}{2n-1}\big)^n \to 0$ and consequently $\varphi(h)=0$. Thus $\varphi(h)\in \{0, 1\}$. }
\end{remark}

Motivated by  Lemmas \ref{Mult-Unif}, \ref{Mult-Lip} and Remark \ref{Rem-1},  we  give the following  definition. 
\begin{definition}
We say that a Banach function algebra $B$ on a compact Hausdorff space $X$ has  Bishop property, if for each point  $x_0\in \ch(B)$  and $f\in B$ with $f(x_0)\neq 0$ there exists a peaking function $g\in P_{x_0}(B)$ such that $g\in  {\rm exp}(B) \cup B_+$  and  $\frac{1}{f(x_0)}fg \in P_{x_0}(B)$.
\end{definition} 
Hence  uniform algebras on a compact Hausdorff space, the Banach function algebras ${\rm Lip}( X)$ for a compact metric space $X$ and ${\rm AC}([0,1])$ have the Bishop property. 

As we stated in Proposition \ref{Main-P} if $\ch(A)=X$, then  for a multiplicative spectral function $\varphi:A \to \Bbb C$ we have either ${\rm ker}(\varphi)$ is a maximal ideal of $A$ (that is ${\rm ker}(\varphi)={\rm ker}(\varphi_{x_0})$, where $x_0$ is the associated point of $\varphi$) or $1=f_1+\cdots +f_n$ for some $f_1,..., f_n\in {\rm ker}(\varphi)$. 

\begin{lem}\label{Main-lem}
Let $\ch(A)=X$ and $A$ have  the Bishop property. Then for any multiplicaive spectral function $\varphi \colon  A \to \Bbb C$ (with the associated point $x_0$)  in  either of cases that  

i)  $\varphi$ is continuous, 

or 

ii) $1 \notin {\rm span}({\rm ker}(\varphi))$,

\noindent 
we have  $\varphi(f)=f(x_0)$ for all $f\in A$ such that  $\overline{f}\in A$. 
 \end{lem}
  \begin{proof}
 We first note that  $\varphi(g)=g(x_0)$ for all $g\in A_+$. Indeed,   for each  $g\in A_+$  and $n\in \Bbb N$ the spectrum of the  function $g_n=g+\frac{1}{n}$  in $A$ is contained in the set of all positive real numbers, and consequently it has a logarithm in $A$. Thus  $\varphi(g_n)=g_n(x_0)=g(x_0)+\frac{1}{n}$. Hence, in the case that  $\varphi$ is continuous we have  $\varphi(g)=\lim_{n \to \infty} \varphi(g_n)=g(x_0)$.
In the second case  $1 \notin {\rm span}({\rm ker}(\varphi))$ and consequently  we have ${\rm ker}(\varphi)={\rm ker}(\varphi_{x_0})$. Thus  Lemma  \ref{lem-real} implies that $\varphi(g)=g(x_0)$. Hence, in both cases  we have $\varphi(g)=g(x_0)$ for all $g\in A_+$.  

 

Now let $f\in A$  such that  $\overline{f} \in A$.  We may assume that $f(x_0)=1$.   By the above argument we have $\varphi(f\overline{f})=1$ and consequently  \[1= |\varphi(f) \varphi(\overline{f})|=|\varphi(f)|\, |\varphi(\overline{f})|\le |f(x_0)|^2=1.\]
Therefore, $|\varphi(f)|=1$. Now choose, by the Bishop property,  a function $h\in  P_{x_0}(A)$ such that $h\in {\rm exp}(A)\cup A_+$ and  $fh\in P_{x_0}(A)$.  By the above argument, we have again $\varphi(h)=1$, hence 
 $|\varphi(f h)|=|\varphi(f) \varphi(h)|=1$ and since $fh(X) \subseteq {\rm int}(\Bbb D)\cup \{1\}$ we get $\varphi(fh)=1$. Thus $1=\varphi(fh)=\varphi(f) \varphi(h)=\varphi(f)$, as desired.  
\end{proof}
Now  we have the following theorem which provides  our first result concerning linearity of multiplicative spectral functions on some natural Banach function algebras having Bishop property (including $C(X)$, Lipschitz algebras of functions and $AC([0,1])$). In the lack of Bishop property,  the second theorem (Theorem \ref{Main2}) states similar results  for natural Banach function algebras on compact metric spaces satisfying certain separation property (including $C^1([0,1])$).  

\begin{theorem}\label{Main1}
Let  $A$ be a natural Banach function algebra on a compact Hausdorff space $X$ with $\ch(A)=X$  and  let  $\varphi:A \to \Bbb C$ be a (not necessarily continuous) multiplicative spectral function.  Then 

(i) Either $\ker(\varphi)$ is a maximal ideal of $A$ or $1=f_1+ \cdots +f_n$ for some $f_1, ..., f_n \in \ker(\varphi)$.  

(ii) Assume, furthermore,  that $A$ has the Bishop property. Then in either of cases that   $1 \notin {\rm span}({\rm ker}(\varphi))$ or $\varphi$ is continuous  
we have  $\varphi(f)=f(x_0)$ for all $f\in A$ such that  $\overline{f}\in A$. 
\end{theorem}

The example given in \cite{Ma} shows that there exists a nonzero multiplicative spectral function $\varphi \colon  C([0,1])\to \Bbb C$ such that ${\rm span}(\ker(\varphi))=C([0,1])$, that is $\ker(\varphi)$ is not a (maximal)  ideal. 

The following  corollary is immediate from the above theorem. As it was mentioned before, the case that $\varphi$ is continuous can also be deduced from  the results of \cite{Brits-CX} and \cite{Brits-Her}.  
\begin{cor}
Let $A$ be either $C(X)$ for some compact Hausdorff space $X$ or one the Banach function algebras ${\rm Lip}(X)$ where $X$ is a compact metric space or ${\rm AC}([0,1])$. Let $\varphi:A\to \Bbb C$ be a multiplicative spectral function on $A$. Then the following are equivalent: 

(i) $\varphi$ is a character. 

(ii) $\varphi$ is continuous. 

(iii) $1 \notin {\rm span}({\rm ker}(\varphi))$. 
\end{cor}
We should note that for a unital Banach algebra $A$, the proof of Theorem \ref{Ma-th} can be applied to show that if $G$ is a multiplicative subgroup of $G(A)$ containing ${\rm exp}(A)$, then for any multiplicative spectral function $\varphi \colon  G \to \Bbb C$ there exists a character $\psi$ on $A$ such that $\varphi(x)=\psi(x)$ for all $x\in G_1(A)$. Then the proof of Proposition \ref{Main-P} (iii) yields the following corollary. 
\begin{cor}
Let  $A$ be a natural uniform algebra on a compact Hausdorff space $X$ with $\ch(A)=X$, $G$ be a subgroup of $G(A)$ containing ${\rm exp}(A)$ and let $\varphi:G \to \Bbb C$ be a multiplicative function  satisfying  $\varphi(f) \in \sigma(f)$ for all $f \in G$. Then $\varphi$ is extended to a character on $A$. 
\end{cor}

In the lack of the Bishop property, next theorem gives a similar result to Theorem \ref{Main1}(ii), under an alternative condition. 
\begin{theorem}\label{Main2}
   Let $A$ be a natural Banach function algebra on a compact metric space $X$ such that for any pair of disjoint closed subsets $F$ and $K$ of $X$ there exists a function $h\in A$ with values in $[0,1]$ satisfying $h=0$ on $F$ and $h=1$ on $K$. Let $\varphi\colon  A \to \Bbb C$ be a multiplicative spectral function. 
   If either $\varphi$ is continuous or $1\notin {\rm span}({\rm ker}(\varphi))$, then  there exists a point $x_0\in X$ such that 
 $\varphi(f)=f(x_0)$ for all $f\in G(A)$ and all $f\in A$ with  $\overline{f}\in A$. 
   \end{theorem}
   \begin{proof}
   We first note that each point $x_0\in X$ is a peak point for $A$. Indeed, since $\{x_0\}=\bigcap_{n=1}^{\infty} U_n$ for some decreasing sequence $\{U_n\}_{n\in \Bbb N}$ of open neighborhoods of $x_0$, it follows from the hypotheses that for each $n\in \Bbb N$, there exists $f_n\in A$ such that $0\le f_n\le 1$, $f_n(x_0)=1$ and $f_n=0$ on $X\backslash U_n$. It is easy to see that for some appropriate $\alpha>0$, the function  $h_0= \alpha \sum_{n=1}^{\infty} \frac{f_n}{2^n \|f_n\|}$ is a (positive) peaking function of $A$ at the point $x_0$. 
 This  implies that each point in $X$ is a peak point  of $A$ and consequently  $\ch(A)=X$.  
 
     Using  Proposition \ref{Main-P}  we obtain a  point $x_0\in X$ such that $\varphi(e^f)=e^{f(x_0)}$ for all $f\in A$,   $|\varphi(f)|\le |f(x_0)|$ for all $f\in A$ and $|\varphi(f)|=|f(x_0)| $ whenever $f\in G(A)$. 
     We also conclude that either ${\rm ker}(\varphi)$ is a maximal ideal of $A$ or $1\in  {\rm span}({\rm ker}(\varphi))$. 
     
      Clearly  if $f\in A$ and $f(x_0)=0$, then $\varphi(f)=0=f(x_0)$.       
    As in the proof of Lemma \ref{Main-lem}  we have $\varphi(g)=g(x_0)$ for all $g\in A_+$.  
     Hence for each $f \in A$ with  $\overline{f}\in A$, we have again $|\varphi(f)|=|f(x_0)|$. Indeed, assuming that  $f(x_0)=1$, since $f\overline{f}\in A_+$   we have 
     \[1= f(x_0)\overline{f(x_0)}=\varphi(f\overline{f})=\varphi(f)\varphi(\overline{f}) \] 
 and consequently $|\varphi(f)|=1$.
 
 In the sequel we show that for each $f\in A$ with $|\varphi(f)|=|f(x_0)|$ we have $\varphi(f)=f(x_0)$, which clearly concludes the desired equality in both cases.  For this let  $f\in A$ and $|\varphi(f)|=|f(x_0)|$. We may assume that $f(x_0)=1$. Then    $\varphi(f)=e^{i \theta}$ for some $\theta \in [0,2\pi)$. Now for each $n \in \Bbb N$, choose, by assumption,  $h_n\in A$ with values in $[0,1]$ such that $h_n=1$ on $\{x\in X\colon  |f(x)|\le 1\}$ and $h_n=0$ on ${\{x\in X\colon  |f(x)|\ge  1+ \frac{1}{n}\} }$. 
  By the above argument, $x_0$ is a peak point of $A$ and the corresponding peaking function $h_0$ can be chosen to be nonnegative. In particular,  we have $\varphi(h_0)=h_0(x_0)=1$. 
   Setting $g_n=h_0h_n$ for $n\in \Bbb N$, we have 
 $\varphi(g_n)=g_n(x_0)=1$, since $g_n$ is also nonnegative. Now it follows from the hypotheses that 
    $$ e^{i \theta}=\varphi(g_n) \varphi(f)=\varphi(g_n f)= g_n(x_n) f(x_n)$$
    for some $x_n \in X$. Thus 
    \begin{equation}\label{eqxn}
e^{i \theta}= g_n(x_n) f(x_n) \qquad (n \in \Bbb N).
\end{equation}
    In the case  that there exists $n_0\in \Bbb N$ such that $|f(x_{n_0})| \le 1$,  we get 
    $$ e^{i \theta}= g_{n_0}(x_{n_0})f(x_{n_0})= h_0(x_{n_0}) f(x_{n_0}),$$ 
    which conclude that $|h_0(x_{n_0}) f(x_{n_0})|=1$, that is $|h_0(x_{n_0})|=1=|f(x_{n_0})|$. In particular, $x_{n_0}=x_0$, since $x_0$ is a peak point of $A$ with the peaking function $h_0$. Thus in this case we have 
    $$ \varphi(f)=e^{i \theta}=g_{n_0}(x_{n_0})f(x_{n_0})=g_{n_0}(x_0)f(x_0)=1,$$ 
    as desired. 
    So we assume  that $|f(x_n)|>1$ for all  $n\in \Bbb N$.  
    Clearly \eqref{eqxn} shows that for each $n\in \Bbb N$ we have $1< |f(x_n)|< 1+\frac{1}{n}$.
    Passing through a subsequence we may assume that $\{x_n\}$ converges to a point $x\in X$. Since $ e^{i \theta} =h_0(x_n) h_n(x_n) f(x_n)$ for each $n\in \Bbb N$, we have 
     $$1= |h_0(x_{n})|\, |h_n(x_n)|\, |f(x_n)| \le  |h_0(x_n)| (1+\frac{1}{n}).$$
    Tending $n$ to infinity it follows that $1\le |h_0(x)|$, that is $|h_0(x)|=1$ and consequently $x=x_0$. 
    Hence, by \eqref{eqxn} we get $\lim g_n(x_n)=e^{i \theta}$ and since for each  $n\in \Bbb N$ we have $g_n\ge 0$, it follows that $\theta=0$. Thus $\varphi(f)=1=f(x_0)$, as desired. 
            \end{proof}
      \begin{cor}
   Let $\varphi\colon  C^1([0,1]) \to \Bbb C$ be  a  multiplicative spectral function. Then the following statements are equivalent: 
    
    (i) $\varphi$ is a character. 
    
    (ii) $\varphi$ is continuous.
    
    (iii) $1\notin {\rm span}({\rm ker}(\varphi)$. 
       \end{cor}
       



   {\bf Acknowledgement.}  This work is based upon  research funded by Iran National Science Foundation (INSF) under project No. 4031504.

   \end{document}